\documentclass[3p,preprint]{elsarticle}

\usepackage{amsmath}
\usepackage{amsfonts}
\usepackage{amssymb,latexsym}
\usepackage{amsthm}
\usepackage{graphicx}
\usepackage{ulem}

\newtheorem{Remark}{Remark}[section]
\newtheorem{Corollary}{Corollary}[section]

\newtheorem{Proposition}{Proposition}[section]

\usepackage{color}
\usepackage{newlfont} 
\usepackage{subfigure}
\usepackage{verbatim}
\usepackage{rotating} 
\usepackage{multirow}
\usepackage{units}
\usepackage{bm}
\usepackage[english]{babel}
\usepackage[utf8]{inputenc}
\usepackage{algorithm}
\usepackage[noend]{algpseudocode} 
\usepackage{booktabs}
\usepackage{caption}
\usepackage{hyperref}
\usepackage{subfigure}
\usepackage[export]{adjustbox}

\journal{...}

\begin{document}

\begin{frontmatter}

\title{On modified anti-Gaussian rules for Jacobi weight functions}

\author[address-TS]{Eleonora Denich}
\ead{eleonora.denich@units.it}
\author[address-TS]{Paolo Novati}
\ead{novati@units.it}
\author[address-PD]{Alvise Sommariva\corref{corrauthor}}
\ead{alvise@math.unipd.it}

\cortext[corrauthor]{Corresponding author}

\address[address-TS]{Department of Mathematics, Informatics and Geosciences, University of Trieste, Via Valerio 12/1, 34127 Trieste, Italy}
\address[address-PD]{Department of Mathematics "Tullio Levi-Civita", University of Padova, Italy}

\begin{abstract}



Anti-Gaussian formulas represent an efficient tool for a dynamical estimation of the error of the underlying Gaussian rule.
When applied to the Jacobi weight function it is known that such formulas are not always internal. 
In this work we show how to overcome this problem by using the so called modified anti-Gaussian rule with \textcolor{black}{suitable parameter $\theta=\theta(n)$, that depends on the number $n$ of quadrature points of the Gaussian formula.}  
Next we study theoretically the asymptotic rate of convergence of the corresponding modified averaged Gaussian formulas. We conclude by showing the benefits of this approach via numerical experiments. All the Matlab codes used in this work are available as open-source software at {\cite{GITDNS}}.

\end{abstract}


\end{frontmatter}

\section{Introduction}
In order to approximate
\begin{equation} \label{integrale I}
    I(f)=\int_a^b f(x) w(x) dx,
\end{equation}
where $w$ is a positive weight function on $[a,b]$ such that all the moments exist and are finite, one may use a $n$-point Gaussian quadrature rule 
$$
G^{(n)} (f)=\sum_{k=1}^n w_k f(x_k),
$$
with nodes $\{x_k\}_{k=1,\ldots,n}$ and weights $\{w_k\}_{k=1,\ldots,n}$.

In \cite{LAU96} the author introduces the $(n+1)$-point anti-Gaussian rule
$$
H^{(n+1)} (f)=\sum_{k=1}^{n+1} {\lambda}_k f({\xi}_k),
$$
that is designed to fulfill the equality 
$$
I(p) -H^{(n+1)} (p)=-(I(p)-G^{(n)}(p)), 
$$
for all $p \in {\mathbb{P}}_{2n+1}$, the set of polynomials of degree at most $2n+1$.
The above relation implies 
\begin{eqnarray}
I( p)=\frac{H^{(n+1)} (p)+G^{(n)}( p)}{2}, \quad p \in {\mathbb{P}}_{2n+1},
\end{eqnarray}
and, consequently, the error
\begin{equation} \label{errore gaussiana}
    R_G^{(n)} (f):=I(f)-G^{(n)}(f )
\end{equation}
is such that
\begin{eqnarray} 
R_G^{(n)} (p)&=& \frac{H^{(n+1)} (p)+G^{(n)}( p)}{2} - G^{(n)} (p) \label{formula 3} \\
&=&   \frac{H^{(n+1)} (p)-G^{(n)} (p)}{2}, \quad p \in {\mathbb{P}}_{2n+1} \label{relazione 2}
\end{eqnarray}
The $(2n+1)$-point formula
$$
A^{(2n+1)} = \frac{H^{(n+1)} +G^{(n)}}{2}
$$
is called averaged Gaussian rule.
Relation (\ref{relazione 2}) suggests that the averaged formula offers a strategy for a-posteriori error estimation of the error of the Gaussian rule.
As for the theoretical properties, it is proved that the anti-Gaussian formula $H^{(n+1)}$ has positive weights, all the nodes are interlaced by those of $G^{(n)}$, the inner nodes $\xi_2,\ldots, \xi_{n}$ belong to $[a,b]$.
One of the issues is that, depending on the weight function $w$, the anti-Gaussian rule may have $\xi_{1}$ and/or $\xi_{n+1}$ external to $[a,b]$.
In particular, in {\cite{LAU96}} the author observes that if we denote by  $\{\phi_k\}_{k=0,\ldots}$ the monic family of orthogonal polynomials with respect to $w$, and consider the three terms recursion
$$
\phi_{n+1}(x)=(x-a_n)\phi_n(x)-b_n \phi_{n-1}(x), \quad \phi_0(x)=1, \, \phi_{-1}(x)=0,
$$
where
$$
a_k=\frac{I(x \phi^2_k)}{I(\phi^2_k)}, \quad b_k=\frac{I(\phi^2_{k})}{I(\phi^2_{k-1})},
$$
then the anti-Gaussian rule $H^{(n+1)}$ is internal if and only if 
\begin{equation}\label{in_ineq}
\frac{\phi_{n+1}(a)}{\phi_{n-1}(a)} \geq b_n, \quad \frac{\phi_{n+1}(b)}{\phi_{n-1}(b)} \geq b_n.
\end{equation} 
In the same paper it is analyzed the case of Jacobi weight $$w_{\alpha,\beta}(x)=(1-x)^{\alpha} (1+x)^{\beta} , \quad \alpha, \beta > -1, \, x \in (-1,1),$$ proving and illustrating for what couple of exponents ($\alpha$, $\beta$) this happens.

Similarly to \cite{EHR02,CR03,HASC08,SP07,DFVI25}, we focus our attention on the modified anti-Gaussian rule $H^{(n+1)}_{\theta}$ that fulfills 
\begin{equation} \label{H_theta intro}
    H^{(n+1)}_{\theta}(p)=(1+\theta) I(p) - \theta G^{(n)}(p), \quad p \in {\mathbb{P}}_{2n+1},
\end{equation}
for some $\theta >0$, and the modified averaged formula
\begin{equation} \label{formula weighted averaged}
    A^{(2n+1)}_{\theta} = \frac{ {1}}{1+\theta}  H_{\theta}^{(n+1)}  +\frac{\theta}{1+\theta} G^{(n)}.
\end{equation}
For $\theta=1$ the above relations reduce to the classical case in which $H^{(n+1)}=H_1^{(n+1)}$, $A^{(2n+1)}=A_1^{(2n+1)}$.

Working with a weight function for which the anti-Gaussian rule $H^{(n+1)}_1$ is not internal, our basic aim is to define $\theta= \theta(n)$ in order that: 
\begin{enumerate}
    \item the corresponding modified anti-Gaussian and averaged rules are internal;
    \item the modified averaged rule is asymptotically faster than the underlying Gaussian rule \cite{DN26}.
\end{enumerate}
Through the paper we use the symbol $\sim$ to denote the asymptotic equality, while $\lesssim$ stands for less than or asymptotically equal to.

This work is organized as follows. 
In Section \ref{section 2} we introduce the modified anti-Gaussian and averaged rules, together with some theoretical properties.
Section \ref{section 3} deals with the choice of the parameter that ensures internality for the Jacobi weight function.
The rate of convergence of the modified averaged rule is studied in Section \ref{section 4} in the case of analytic functions. Finally, in Section \ref{section 5} we give the explicit a-priori error representation of the modified averaged rule, together with some numerical experiments. 




\section{Modified anti-Gaussian rule} \label{section 2}

We start our analysis by introducing the $(n+1)$-point modified anti-Gaussian rule
$$
H_{\theta}^{(n+1)} (f)=\sum_{k=1}^{n+1} {\lambda}_k f({\xi}_k), \quad \theta >0,
$$
where $\left\lbrace\lambda_k\right\rbrace_{k=1,\ldots,n+1}$, $\xi_1<\ldots<\xi_{n+1}$ are the weights and nodes, respectively.
In order to keep the notations as simple as possible, wo omit the dependency on $\theta$ when unnecessary.
The above formula is such that (cf. (\ref{H_theta intro}))
$$
I(p) -H_{\theta}^{(n+1)}( p)=-\theta \left(I( p)-G^{(n)} (p) \right), \quad p \in {\mathbb{P}}_{2n+1}. 
$$
For convenience, we prove some relevant results that are similar to those presented in {\cite{SP07}}. 
As in {\cite{LAU96}}, via a minor modification, one gets that, for $p \in {\mathbb{P}}_{2n+1}$,
\begin{eqnarray} \label{formula 4}
I( p)=A_{\theta}^{(2n+1)}(p)=\frac{ {1}}{1+\theta}  H_{\theta}^{(n+1)} (p) +\frac{\theta}{1+\theta} G^{(n)}( p),
\end{eqnarray}
that is, the integral is a convex combination of $H_{\theta}^{(n+1)} $ and $G^{(n)}$. We immediately get
$$
H^{(n+1)}_{\theta}(p)=(1+\theta) I(p) - \theta G^{(n)}(p) , \quad p \in {\mathbb{P}}_{2n+1},
$$
so that,  $H_{\theta}^{(n+1)}$ is the $(n+1)$-point Gaussian rule for the linear functional
$$
{\cal{L}}^{(\theta)}_{n} = (1+\theta) I - \theta G^{(n)}.
$$
Defining by $\{\pi_{j, \theta}\}_{j\geq 0}$ the set of monic orthogonal polynomials with respect to the functional  ${\cal{L}}^{(\theta)}_{n}$, we have
$$
\left\{
\begin{array}{l}
\pi_{n+1,\theta}(x)=(x-\alpha^{(\theta)}_n)\pi_{n,\theta}(x)-\beta^{(\theta)}_n \pi_{n-1,\theta}(x), \quad n \geq 0\\
\pi_{-1,\theta}(x)=0, \, \pi_{0,\theta}(x)=1,
\end{array}
\right.
$$
where
$$
\alpha^{(\theta)}_k=\frac{{\cal{L}}^{(\theta)}_{n}(x \pi^2_{k,\theta})}{{\cal{L}}^{(\theta)}_{n}(\pi^2_{k,\theta})}, \quad \beta^{(\theta)}_k=\frac{{\cal{L}}^{(\theta)}_{n}(\pi^2_{k,\theta})}{{\cal{L}}^{(\theta)}_{n}(\pi^2_{k-1,\theta})}.
$$
Since the Gaussian rule has algebraic degree of exactness $2n-1$, for $p \in {\mathbb{P}}_{2n-1}$
$$
 {\cal{L}}^{(\theta)}_{n}(p)=(1+\theta) I(p)-\theta G^{(n)}(p)= I(p).
$$
We observe that, for $j,k=0,\ldots,n$, $j \neq k$,
$$
{\cal{L}}^{(\theta)}_{n}(\phi_j \phi_k)= I(\phi_j \phi_k)=0,
$$
so that, by uniqueness of monic orthogonal polynomials, for $k=0,\ldots,n$, we have $\phi_k=\pi_{k,\theta}$ .
As consequence, for $k=0,\ldots,n-1$
$$
\alpha^{(\theta)}_k=\frac{{\cal{L}}^{(\theta)}_{n}(x \pi^2_{k,\theta})}{{\cal{L}}^{(\theta)}_{n}(\pi^2_{k,\theta})}=\frac{ I(x \pi^2_{k,\theta})}{ I(\pi^2_{k,\theta})}=\frac{ I(x \phi^2_k)}{ I(\phi^2_k)}=a_k,
$$

$$
\beta^{(\theta)}_k=\frac{{\cal{L}}^{(\theta)}_{n}(\pi^2_{k,\theta})}{{\cal{L}}^{(\theta)}_{n}(\pi^2_{k-1,\theta})}=\frac{ I( \pi^2_{k,\theta})}{ I(\pi^2_{k-1,\theta})}=\frac{ I(\phi^2_{k})}{ I(\phi^2_{k-1})}=b_k.
$$

Now, observe that, since $\phi_n = \pi_{n,\theta}$ and $x_1,\ldots,x_n$ are zeros of $\phi_n$, necessarily
$$
G^{(n)}(x \pi^2_{n,\theta})=\sum_{k=1}^n w_k x_k  \pi_{n,\theta}(x_k)=\sum_{k=1}^n w_k x_k  \phi_n(x_k)=0
$$
and thus
\begin{equation}\label{eq1}
{\cal{L}}^{(\theta)}_{n}(x \pi^2_{n,\theta})=(1+\theta) I(x \pi^2_{n,\theta})-\theta G^{(n)}(x \pi^2_{n,\theta}) =(1+\theta) I(x \pi^2_{n,\theta})=(1+\theta) I(x \phi^2_n).
\end{equation}
Next, for similar reasons,
$$
G^{(n)}(\pi^2_{n,\theta})=\sum_{k=1}^n w_k  \pi_{n,\theta}(x_k)=\sum_{k=1}^n w_k  \phi_n(x_k)=0,
$$
which implies that
\begin{equation}\label{eq2}
{\cal{L}}^{(\theta)}_{n}(\pi^2_{n,\theta})=(1+\theta) I(\pi^2_{n,\theta})-\theta G^{(n)}( \pi^2_{n,\theta}) =(1+\theta) I( \phi^2_n).
\end{equation}
In view of (\ref{eq1}) and (\ref{eq2})
$$
\alpha^{(\theta)}_n=\frac{{\cal{L}}^{(\theta)}_{n}(x \pi^2_{n,\theta})}{{\cal{L}}^{(\theta)}_{n}(\pi^2_{n,\theta})}=
\frac{(1+\theta) I(x \phi^2_n)}{(1+\theta) I( \phi^2_n)}=\frac{ I(x \phi^2_n)}{I( \phi^2_n)}=a_n,
$$
and, from (\ref{eq2}),
\begin{equation}\label{eq3}
\beta^{(\theta)}_n=\frac{{\cal{L}}^{(\theta)}_{n}(\pi^2_{n,\theta})}{{\cal{L}}^{(\theta)}_{n}(\pi^2_{n-1,\theta})}=
\frac{(1+\theta) I( \phi^2_n)}{I( \phi^2_{n-1})}=(1+\theta) b_n.
\end{equation}
The case $\theta=1$ gives the classical result of anti-Gaussian rules in which $\beta^{(1)}_n=2b_n$.
We also remark that, by choosing $\theta = \frac{b_{n+1}}{b_{n}}$ in (\ref{formula 4}) we obtain the so called {\it{generalized averaged Gaussian rule}} (see \cite{RSP21}).

At this point it is important to observe that, in view of these results,
\begin{align} \label{7bis}
    \pi_{n+1,\theta}(x)&=(x-a_n) \phi_n(x)-\left(1+\theta\right) b_n \phi_{n-1}(x) \notag \\
    &=\phi_{n+1}(x)-\theta b_n \phi_{n-1}(x).
\end{align}
The largest node $\xi_{n+1}$ of the rule $H^{(n+1)}_{\theta}$ belongs to the interval $[a,b]$ if and only if $\pi_{n+1,\theta}(b) \geq 0$, that is,
\begin{equation} \label{condizione}
    \phi_{n+1}(b)-\theta b_n \phi_{n-1}(b) \geq 0,
\end{equation}
or, alternatively,
\begin{equation}\label{eq4}
\frac{\phi_{n+1}(b)}   {\phi_{n-1}(b)}  \geq \theta b_n.
\end{equation}
By similar reasoning, the condition for $\xi_1 \in [a,b]$ is 
 \begin{equation}\label{eq5}
\frac{\phi_{n+1}(a)}   {\phi_{n-1}(a)}  \geq \theta b_n.
\end{equation}
A key point for our next analysis is that, in view of (\ref{eq4})-(\ref{eq5}), suitably selecting $\theta$, the modified anti-Gaussian rule is internal.
Moreover, independently of $\theta$, the modified anti-Gaussian rule still provides a reliable error estimate, since (cf. (\ref{errore gaussiana}))
\begin{equation} \label{stimatore}
E_G^{(n)}(p) = \frac{1}{1+\theta}(H_{\theta}^{n+1}(p)-G^{(n)}(p)), \quad p \in \mathbb{P}_{2n+1}. 
\end{equation}
With a proper choice of $\theta =\theta(n)$, we will see that the modified averaged rule $A_{\theta}^{(2n+1)}$ offers not only a way to automatically estimate the error of the Gaussian rule, but also exhibits a faster rate of convergence, similarly to what happens for  the classical averaged rule $A_1^{(2n+1)}$ (see \cite{DN26,RSP26}).


\begin{Remark}
By means of Cauchy interlace theorem, independently of $\theta$, the nodes $\xi_i$, $i = 1, 2,...,n + 1$, are all real, and are interlaced by those of
the Gaussian formula $G^{(n)}$, i.e.,
$$
\xi_1 < x_1 < \xi_2 < x_2 < \ldots < x_n < \xi_{n+1}.
$$
\end{Remark}

\section{Modified anti-Gaussian rule for Jacobi weight functions} \label{section 3}

In the case of Jacobi weight function $w_{\alpha,\beta}(x)=(1-x)^{\alpha} (1+x)^{\beta}$ with $\alpha$, $\beta > -1$, $x \in (-1,1)$, for the related monic orthogonal polynomials we have (see \cite[18.6.1]{NIST})
\begin{align}
    \phi_{n}^{(\alpha,\beta)}(1) &= \frac{2^n \binom{n+\alpha}{n}}{\binom{2n+\alpha+\beta}{n}}, \label{relA} \\
    \phi_{n}^{(\alpha,\beta)}(-1) &=(-1)^n \frac{2^n \binom{n+\beta}{n}}{\binom{2n+\alpha+\beta}{n}}, \label{relB}
\end{align}
and, moreover,
\begin{equation} \label{relC}
    b_n = \frac{4n(n+\alpha)(n+\beta)(n+\alpha+\beta)}{(2n+\alpha+\beta-1)(2n+\alpha+\beta)^2(2n+\alpha+\beta+1)}.
\end{equation}
Defining the function 
\begin{equation} \label{funzione h}
h(n,x,y) = 1+\frac{(2x+1)n(n+x+y+1)+\frac{1}{2} (x+1)(x+y)(x+y+1)}{n(n+y)(n+\frac{x+y}{2}+1)},
\end{equation}
by (\ref{relA})-(\ref{relB})-(\ref{relC}) we obtain, after some computations,
\begin{align*}
    \frac{\phi_{n+1}^{(\alpha,\beta)}(1)}{b_n \phi_{n-1}^{(\alpha,\beta)}(1)}&=h(n,\alpha,\beta), \\
    \frac{\phi_{n+1}^{(\alpha,\beta)}(-1)}{b_n \phi_{n-1}^{(\alpha,\beta)}(-1)}&=h(n,\beta,\alpha).
\end{align*}
Thus, by (\ref{eq4})-(\ref{eq5}) the modified anti-Gaussian rule $H_{\theta}^{(n)}$ is internal if and only if
\begin{equation} \label{eq7j}
    \theta \leq \theta_{\mbox{\footnotesize{max}}}(n,\alpha,\beta):=\min \left\lbrace {h(n,\alpha,\beta)}, {h(n,\beta,\alpha)} \right\rbrace.
\end{equation}
From (\ref{funzione h}), we observe that $\theta_{\mbox{\footnotesize{max}}}(n,\alpha,\beta) \rightarrow 1$ as $n \rightarrow +\infty$.

In the case of ultraspherical (Gegenbauer) polynomials, that corresponds to the choice $\alpha=\beta=\gamma-1/2$ with $\gamma > -1/2$, we have, in particular,  that the nodes $\xi_1, \xi_{n+1} \in [-1,1]$ if and only if 
\begin{equation}\label{eq7}
\theta\leq \theta_{\mbox{\footnotesize{max}}}(n,\gamma):=1+  \frac{2\gamma n(n+2\gamma) +(\gamma+1/2)(2\gamma-1)\gamma }{n(n+\gamma-1/2)(n+\gamma+1/2)} 
\end{equation}
In view of (\ref{eq7}), we have $\theta_{\mbox{\footnotesize{max}}}(n,\gamma) \rightarrow 1$ as $n \rightarrow +\infty$, from below when $\gamma < 0$ and from above when $\gamma > 0$ (see Figure \ref{jacobi_thetamin}).
Note that, for $\gamma \geq 0$, $\theta_{\mbox{\footnotesize{max}}}(n,\gamma) \geq 1$, and the anti-Gaussian formula is already known to be internal, so that the use of the modified rule is not necessary.





Moreover, by (\ref{condizione}), taking $\theta=\theta_{\mbox{\footnotesize{max}}}(n,\gamma)$ the corresponding modified anti-Gaussian rule is just the Gegenbauer-Lobatto formula, since $\xi_1=-1$, $\xi_{n+1}=1$ and, by construction, the rule has degree $2n-1$. This property was already stated in \cite[Th.3.6]{CR03}.



\begin{figure}[ht!]
\centering
\includegraphics[scale=0.5,valign=c]{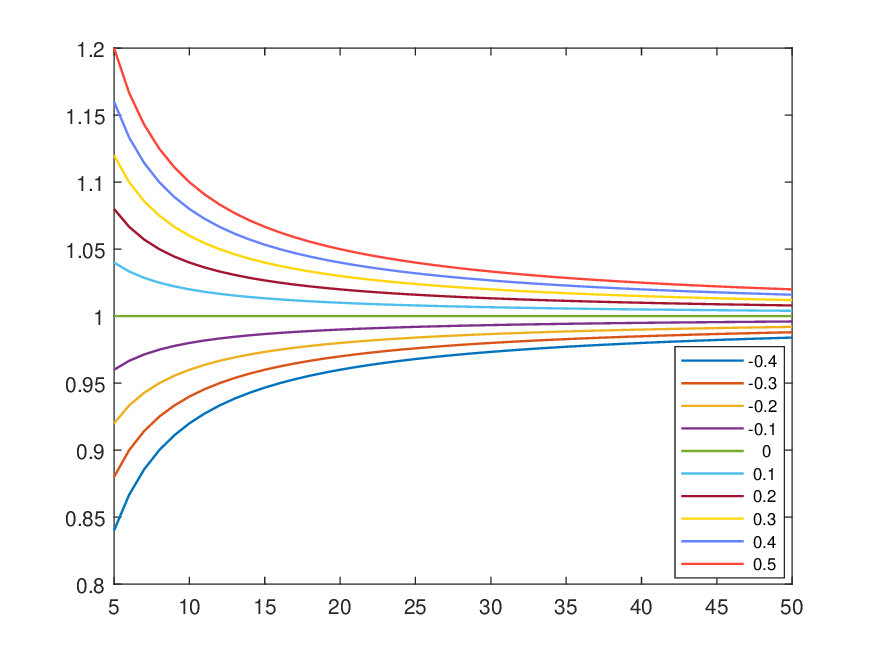} 
\caption{{\color{black}The value $\theta_{\mbox{\footnotesize{max}}}(n,\gamma)$ for $n=5,6\ldots,50$ and $\gamma=-0.5,-0.4,\ldots,0.4,0.5$. We observe that for $\gamma=0$ and $\gamma=0.5$, we consider respectively the Chebyshev and the Legendre weight function and it is known that the anti-Gaussian rule has internal nodes, while this does not hold true for $\gamma < 0$.}}
\label{jacobi_thetamin}
\end{figure}

\section{Asymptotic rate of convergence of the modified averaged rule} \label{section 4}

In this section we investigate the theoretical rate of convergence of the modified anti-Gaussian and averaged formulas, when defining $\theta = \theta_{\mbox{\footnotesize{max}}}(n,\alpha,\beta)$.

Let 
$$
R_{H_{\theta}}^{(n+1)} (f) = I(f) -H_{\theta}^{(n+1)}(f)
$$
be the remainder term of the modified anti-Gaussian rule.
If $H_{\theta}^{(n+1)}$ is internal, $R_{H_{\theta}}^{(n+1)}$ admits the Cauchy integral representation
$$
R_{H_{\theta}}^{(n+1)}(f)= \frac{1}{2\pi i} \int_{\Gamma}K^{(n+1)}_{\theta}(z) f(z) dz,
$$
where the contour $\Gamma$ must contain the interval $[-1,1]$ and the function $f$ has to be analytic in its interior and continuous on the boundary (see \cite{BAR61}).
The kernel $K^{(n+1)}_{H_{\theta}}$ is defined as 
\begin{equation} \label{definizione kernel theta}
    K^{(n+1)}_{H_{\theta}}(z) = \frac{q_{n+1,\theta}(z)}{\pi_{n+1,\theta}(z)}, \quad z \notin [-1,1],
\end{equation}
where $\pi_{n+1,\theta}$ is as in (\ref{7bis}), and
$$
q_{n+1,\theta}(z) = \int_{-1}^1 \frac{\pi_{n+1,\theta}(z)}{z-x}w_{\alpha,\beta}(x) dx
$$
is known as associated function.
For the remainder term $R_G^{(n)}(f) = I(f)-G^{(n)}(f)$ of the Gaussian rule, we have the analog integral representation 
$$
R_{n}^{G}(f)= \frac{1}{2\pi i} \int_{\Gamma}K^{(n)}_{G}(z) f(z) dz,
$$
with
$$
K^{(n)}_{G}(z) = \frac{q_{n}(z)}{\phi_{n}(z)}, \; q_n(z)=\int_{-1}^1 \frac{\phi_{n}(z)}{z-x}w_{\alpha,\beta}(x) dx, \; z \notin [-1,1].
$$
Therefore, since by (\ref{7bis}) we have that 
$\pi_{n+1,\theta}(x)=\phi_{n+1}(x)-\theta b_n \phi_{n-1}(x)$, one obtains 
\begin{eqnarray}
q_{n+1,\theta}(z) &=& \int_{-1}^1 \frac{\pi_{n+1,\theta}(z)}{z-x}w_{\alpha,\beta}(x) dx \nonumber \\
&=& \int_{-1}^1 \frac{\phi_{n+1}(x)-\theta b_n \phi_{n-1}(x)}{z-x}w_{\alpha,\beta}(x) dx \nonumber \\
&=& q_{n+1}(z)-{\theta}b_n q_{n-1}(z),
\end{eqnarray}
and the kernel of the modified anti-Gaussian rule can be written in terms of monic orthogonal polynomials and associated functions of the Gaussian formula as (cf. (\ref{definizione kernel theta}))
\begin{equation} \label{K_theta 25 bis}
    K^{(n+1)}_{H_{\theta}} (z)= \frac{q_{n+1}(z)-{\theta}b_n q_{n-1}(z)}{\phi_{n+1}(z)-{\theta}b_n \phi_{n-1}(z)}, \quad z \notin [-1,1].
\end{equation}
As for the modified averaged formula
\begin{equation} \label{formula weighted averaged REP}
    A^{(2n+1)}_{\theta} = \frac{ {1}}{1+\theta}  H_{\theta}^{(n+1)}  +\frac{\theta}{1+\theta} G^{(n)}.
\end{equation}
{\textcolor{black}{introducing the remainder $R^{(2n+1)}_{A_{\theta}}(f)  = I(f)-A^{(2n+1)}_{\theta}(f)$}}, we still have the Cauchy integral representation 
\begin{equation} \label{rappr int aver}
 R^{(2n+1)}_{A_{\theta}}(f) = \frac{1}{2 \pi i} \int_{\Gamma} K_{A_{\theta}}^{(2n+1)}(z) f(z) dz,   
\end{equation}
where the kernel $ K_{A_{\theta}}^{(2n+1)}$ can be written as
$$
K_{A_{\theta}}^{(2n+1)} = \frac{1}{1+\theta} K_{H_{\theta}}^{(n+1)}(z)+\frac{\theta}{1+\theta} K_{G}^{(n)}(z).
$$
In \cite{DN26}, starting from relation (\ref{K_theta 25 bis}) for $K_{H_1}^{(n+1)}$ and by exploiting for $n \rightarrow +\infty$ suitable asymptotics for $\phi_k$ and $q_k$, the authors derived asymptotic estimates for the remainder terms of the anti-Gaussian and averaged formulas. 
In the following proposition we show how these results can be extended in the case of the modified rules.
First of all, for $n\rightarrow + \infty$, the kernel of the Gauss-Jacobi formula has the asymptotic (see \cite{BAR61},\cite[Sect. 1.13]{DR84},\cite{DE72})
\begin{equation} \label{K_G asintotico}
    K_G^{(n)}(z) \sim 2 \pi w_{\alpha,\beta}(z) \rho(z)^{-2n-\alpha-\beta-1},
\end{equation}
with $\rho(z)=z+\sqrt{z^2-1}$ (the sign of the square root must be taken such that $\vert \rho(z) \vert >1$).

\begin{Proposition} 
    Let $\alpha,\beta$ be such that $H_1^{(n+1)}$ is not internal and $\theta(n)=\theta_{\mbox{\footnotesize{max}}}(n,\alpha,\beta)$ as in (\ref{eq7j}). For $z$ not in the neighborhood of $[-1,1]$ and $n \rightarrow + \infty$ it holds
    \begin{align}
        K_{H_{\theta(n)}}^{(n+1)}(z) &\sim -K_G^{(n)}(z) \left(1-g(n) \frac{1+\rho(z)^2}{1-\rho(z)^2} \right), \label{K_theta asintotico prop} \\
         K_{A_{\theta(n)}}^{(2n+1)}(z) & \sim K_G^{(n)}(z) \frac{g(n)}{1-\rho(z)^2}, \label{K_A_theta asintotico prop}
    \end{align}
    where $g(n) = \theta(n)-1$ is such that
    \begin{equation} \label{g(n)}
        g(n) \sim
        \begin{cases}
            \frac{2\alpha+1}{n}, \quad \alpha \leq \beta \\
            \frac{2\beta+1}{n}, \quad \alpha \geq \beta
        \end{cases}.
    \end{equation}
\end{Proposition}
\begin{proof}
    Following \cite[Prop.4]{DN26}, for $n \rightarrow +\infty$ the kernel of the modified anti-Gaussian rule can be written as
    \begin{equation*}
        K_{\theta(n)}^{(n+1)}(z) \sim -K_G^{(n)}(z) \frac{\eta(n)-{\theta(n)}\psi(n)\rho(z)^2}{\theta(n)\tau(n)-\psi(n)\rho(z)^2},
    \end{equation*}
    where
    \begin{align*}
        \eta(n) &= \frac{4(n+\alpha+1)(n+\beta+1)}{(2n+\alpha+\beta+2)(2n+\alpha+\beta+3)}, \\
        \psi(n) &= \frac{(2n+\alpha+\beta+1)(2n+\alpha+\beta+2)}{4(n+1)(n+\alpha+\beta+1)}, \\
        \tau(n) &= \frac{4n(n+\alpha)(n+\beta)(2n+\alpha+\beta+2)(n+\alpha+\beta)}{(2n+\alpha+\beta)(n+1)(n+\alpha+\beta+1)(2n+\alpha+\beta+1)(2n+\alpha+\beta)},
    \end{align*}
    and $\rho(z)=z+\sqrt{z^2-1}$. 
    Since
    \begin{equation} \label{theta max asintotico}
        \theta(n) = 1+g(n), 
    \end{equation}
    we obtain
    \begin{equation} \label{F}
        K_{\theta(n)}^{(n+1)}(z) \sim -K_G^{(n)}(z) \left(1+\frac{\eta(n)-\tau(n)}{\tau(n)(1+g(n))-\psi(n)\rho(z)^2}-\frac{g(n)(\psi(n)\rho(z)^2+\tau(n))}{\tau(n)(1+g(n))-\psi(n)\rho(z)^2} \right).
    \end{equation}
    By (\ref{funzione h})-(\ref{eq7j}), we have that $g(n) \rightarrow 0$, $n \rightarrow + \infty$, as in (\ref{g(n)}).
    \textcolor{black}{With some effort, it can be seen that}
    \begin{equation} \label{F1}
    \frac{\eta(n)-\tau(n)}{\tau(n)(1+g(n))-\psi(n)\rho(z)^2} = \mathcal{O}\left( \frac{1}{n^3} \right),  \quad  n \rightarrow + \infty,
    \end{equation}
    and, therefore, 
    $$
K_{\theta(n)}^{(n+1)}(z) \sim -K_G^{(n)} \left( 1-g(n) \frac{\psi(n)\rho(z)^2+\tau(n)}{\tau(n)(1+g(n))-\psi(n)\rho(z)^2} \right).
    $$
    Then, by using the asymptotics
    $$
    \psi(n) \sim \tau(n) \sim 1+g(n) \sim 1, \quad n \rightarrow + \infty,
    $$
    we get
$$K_{\theta(n)}^{(n+1)}(z) \sim -K_G^{(n)}(z) \left(1-g(n) \frac{1+\rho(z)^2}{1-\rho(z)^2} \right) $$
    
    \noindent that is, (\ref{K_theta asintotico prop}) holds.
    
    As for the kernel of the modified averaged rule, by means of (\ref{K_theta asintotico prop}) and (\ref{theta max asintotico}), we can write
    \begin{align*}
        K_{A_{\theta(n)}}^{(2n+1)}(z) &= \frac{\theta(n)}{1+\theta(n)}K_G^{(n)}(z)+\frac{1}{1+\theta(n)}K_{\theta(n)}^{(n+1)}(z) \\
        &\sim K_G^{(n)}(z) \left( \frac{\theta(n)}{1+\theta(n)}-\frac{1}{1+\theta(n)}+\frac{g(n)}{1+\theta(n)}\frac{1+\rho(z)^2}{1-\rho(z)^2} \right) \\
        &\sim K_G^{(n)}(z) \, \frac{ \, g(n)}{1+\theta(n)} \left(1+     
        \frac{1+\rho(z)^2}{1-\rho(z)^2} \right)  
        \sim K_G^{(n)}(z) \frac{ \, g(n)}{2+g(n)} 
        \frac{2}{1-\rho(z)^2} \\
        & = K_G^{(n)}(z)  \frac{g(n)}{1-\rho(z)^2}.
    \end{align*}
\end{proof}



\begin{Corollary}
    Assuming $-1/2<\gamma<0$ and defining $\theta(n)=\theta_{\mbox{\footnotesize{max}}}(n,\gamma)$, for z not in the neighborhood of $[-1,1]$ and $n \rightarrow + \infty$, it holds
    \begin{align*}
        K_{\theta(n)}^{(n+1)}(z) &\sim -K_G^{(n)}(z) \left(1-\frac{2 \gamma}{n} \frac{1+\rho(z)^2}{1-\rho(z)^2} \right),  \\
         K_{A_{\theta(n)}}^{(2n+1)}(z) & \sim 2K_G^{(n)}(z) \frac{\gamma}{n(1-\rho(z)^2)}. 
    \end{align*}
\end{Corollary}
\begin{Remark}
    Formula (\ref{F}) offers a complete view around the speed up of the averaged rule. Indeed, when $H_1^{(n+1)}$ is internal, $g(n)=0$ and the speed up is given by (\ref{F1}) (see \cite{DN26}). 
    On the other hand, by taking $\theta<1$ constant, $g(n)$ is also constant and $K_{A_{\theta}}^{(2n+1)}\sim K_G^{(n)}$ (no speed up).
\end{Remark}

\section{Error representation and numerical experiments} \label{section 5}

Before presenting some numerical experiments, we show how to exploit the results of the previous section.
{\textcolor{black}{While in \cite{DN26} the authors investigated the case of the classical anti-Gaussian formula for Jacobi weights when all their nodes are internal, here we extend the study to modified average rules even when $\theta_{\mbox{\footnotesize{max}}}(n,\alpha,\beta) < 1$.
Typical instances are ultraspherical weight functions with $\gamma \in (-1/2,0)$.
}}

For simplicity, in what follows $\theta = \theta(n) = \theta_{\mbox{\footnotesize{max}}}(n,\alpha,\beta)$.
We assume $f$ in (\ref{integrale I}) to be a meromorphic function whose restriction to the real numbers is real valued.
With some effort, the analysis given below can be extended to more general functions \textcolor{black}{see e.g. \cite{DN25})}.
The only necessary hypothesis is that $f$ must be analytic in an open subset containing $[-1,1]$.
Now, we introduce confocal ellipses of type
$$
\Gamma_r = \left\lbrace z \in \mathbb{C} \colon z = \frac{1}{2} \left( re^{i \varphi}+\frac{1}{re^{i \varphi}} \right), \; 0 \leq \varphi < 2 \pi \right\rbrace, \quad r>1,
$$
having foci at $\pm 1$ and sum of semiaxis equal to $r$.
We notice that 
\begin{equation} \label{proprietà R}
\vert \rho(z) \vert = \vert z +\sqrt{z^2-1} \vert = r, \quad {\mbox{for}} \; z \in \Gamma_r.
\end{equation}
Suppose that $f$ has no singularity within or on a particular ellipse $\Gamma_{{R}}$, except for a pair of simple poles $z_0$ and its conjugate $\overline{z_0}$, belonging to $\Gamma_{R_0}$, $R_0 < {R}$.
Denoting by $C_1$ and $C_2$ two arbitrary small circles surrounding the two poles, as contour in (\ref{rappr int aver}) we take $\Gamma = \Gamma_{{R}} \cup C_1 \cup C_2$.
By running $\Gamma$ in counterclockwise direction, we have
\begin{eqnarray}
R^{(2n+1)}_{A_{\theta}}(f) 
&=& \frac{1}{2\pi i} \left( \int_{\Gamma_{{R}}} K_{A_{\theta}}^{(2n+1)}(z) f(z) dz - \int_{C_1 \cup C_2} K_{A_{\theta}}^{(2n+1)}(z) f(z) dz \right).
\end{eqnarray}
By using the asymptotics (\ref{K_G asintotico})-(\ref{K_A_theta asintotico prop}) and (\ref{proprietà R}), for $n \rightarrow + \infty$ the contribution on $\Gamma_{{R}}$ is asymptotically bounded by
\begin{align*}
    \left\vert \frac{1}{2\pi i}  \int_{\Gamma_{{R}}} K_{A_{\theta}}^{(2n+1)}(z) f(z) dz \right\vert &\sim \left\vert \frac{1}{2 \pi i} \int_{\Gamma_{{R}}}K_G^{(n)}(z) \frac{g(n)}{1-\rho(z)^2}f(z) dz \right\vert \\
    &\lesssim \vert g(n) \vert R^{-2n-\alpha-\beta-1} \int_{\Gamma_{{R}}} \left\vert \frac{w_{\alpha,\beta}(z)}{1-\rho(z)^2}f(z) dz \right\vert.
\end{align*}
Then, by using the residue theorem, \textcolor{black}{being $f$ a real valued function in $(-1,1)$}, from (\ref{K_A_theta asintotico prop})-(\ref{K_G asintotico}) we get
\begin{align*}
    \frac{1}{2\pi i} \int_{C_1 \cup C_2}  K_{A_{\theta}}^{(2n+1)}(z) f(z) dz &={\mbox Res}(f,z_0)  K_{A_{\theta}}^{(2n+1)}(z_0)+{\mbox Res}(f,\overline{z_0})  K_{A_{\theta}}^{(2n+1)}(\overline{z_0}) \\
    &= 2 \Re \left\lbrace {\mbox Res}(f,z_0) K_{A_{\theta}}^{(2n+1)}(z_0)  \right\rbrace \\
    &\sim 2 \Re \left\lbrace {\mbox Res}(f,z_0)K_G^{(n)}(z_0) \frac{g(n)}{1-\rho({\textcolor{black}{z_0}})^2} \right\rbrace \\
    &\sim 2 \Re \left\lbrace 2 \pi {\mbox Res}(f,z_0) \frac{g(n)}{1-\rho({\textcolor{black}{z_0}})^2}w_{\alpha,\beta}(z_0) \rho(z_0)^{-2n-\alpha-\beta-1} \right\rbrace,
\end{align*}
where the symbol $\Re(\cdot)$ denotes the real part and ${\mbox Res}(f,z_0)$ is the residue of $f$ at $z_0$.



Now, since $R_0<{R}$, we have (cf. (\ref{proprietà R}))
$$
\vert \rho(z_0) \vert^{-2n-\alpha-\beta-1}=R_0^{-2n-\alpha-\beta-1} > {R}^{-2n-\alpha-\beta-1},
$$
so that the contribution on $\Gamma_{{R}}$ can be neglected. Therefore, we consider the asymptotic bound
\begin{equation} \label{bound errore}
    \left\vert R^{(2n+1)}_{A_{\theta}}(f) \right\vert \lesssim 4 \pi \left\vert {\mbox{Res}}(f,z_0) w_{\alpha,\beta}(z_0) \frac{1+\rho({\textcolor{black}{z_0}})^2}{1-\rho({\textcolor{black}{z_0}})^2} g(n) \right\vert R_0^{-(2n+\alpha+\beta+1)}, \quad n \rightarrow + \infty.
\end{equation} 
For the particular case of the Gegenbauer weight function, the error is asymptotically bounded by formula (\ref{bound errore}) with $g(n) \sim -\frac{2 \gamma}{n}$ and $\alpha=\beta=\gamma-\frac{1}{2}$.

In the experiments we consider the following meromorphic functions:
\begin{enumerate}
    \item $f_1(x)= \frac{1}{1+c^2x^2}$, $c \in \mathbb{R}$, with poles $\pm \frac{i}{c}$ and residues $\frac{1}{2ic}$;
    \item $f_2(x)=\frac{1}{1+e^{dx+1}}$, $d \in \mathbb{R}$, with poles $z_k=\frac{1}{d}(-1+i(2k+1)\pi)$, $k \in \mathbb{Z}$, and residues $-\frac{1}{d}$.
\end{enumerate}

\textcolor{black}{Referring to formula (\ref{bound errore}), we have that $R_0=\vert\rho(z_0)\vert$, with $z_0=\frac{i}{c}$ for the function $f_1$, and $z_0=\frac{1}{d}(-1+i\pi)$ in the case of $f_2$.}
In Figures \ref{fig1_weighted} and \ref{fig2_weighted} we plot the error of the Gaussian rule and of the modified averaged rule (with respect to a reference solution). 
Moreover, we also show the a-posteriori error estimate for the Gaussian quadrature, given by (see (\ref{stimatore}))
\begin{equation} \label{stima a posteriori}
R_G^{(n)}(f) \approx \frac{\theta}{1+\theta} \left( H_{\theta}^{(n+1)}(f)-G^{(n)}(f) \right),    
\end{equation}
and the a-priori asymptotic estimate (\ref{bound errore}) for the modified averaged rule.
Figure \ref{fig1_weighted} refers to the Gegenbauer weight, while Figure \ref{fig2_weighted} to the Jacobi case with $\alpha\neq \beta$.
Clearly, we consider values of $\alpha,\beta,\gamma$ such that the corresponding classical anti-Gaussian formula $H_1^{(n+1)}$ is not internal.

\textcolor{black}{In these tests, as manifest in the figures, the a-posteriori estimates of the integration errors of the Gaussian rule and the a-priori estimates of the modified averaged Gaussian rule are rather sharp.}

All the Matlab codes used in this work are available as open-source software at {\cite{GITDNS}}.



\begin{figure}[ht!]
\centering
\includegraphics[scale=0.3,valign=c]{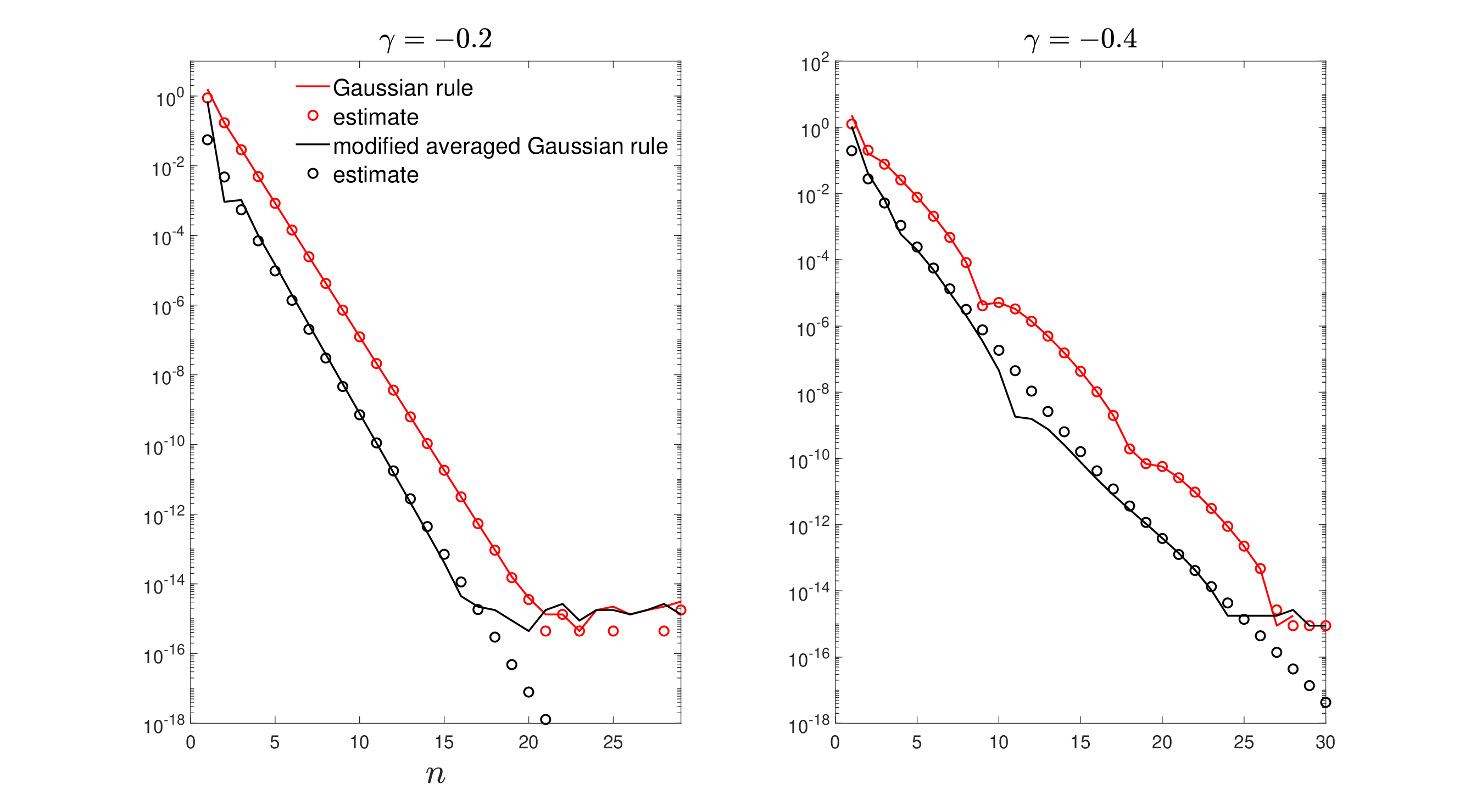} 
\caption{The error of the Gaussian rule (red line) with the a-posteriori estimate (\ref{stima a posteriori}) (red circles), and the error of the modified averaged Gaussian rule (black lines) with the a-priori estimate (\ref{bound errore}) (black circles), for the Gegenbauer weight. The functions considered are $f_1$ with $c=1$ (left) and $f_2$ with $d=5$ (right).}
\label{fig1_weighted}
\end{figure}

\begin{figure}[ht!]
\centering
\includegraphics[scale=0.3,valign=c]{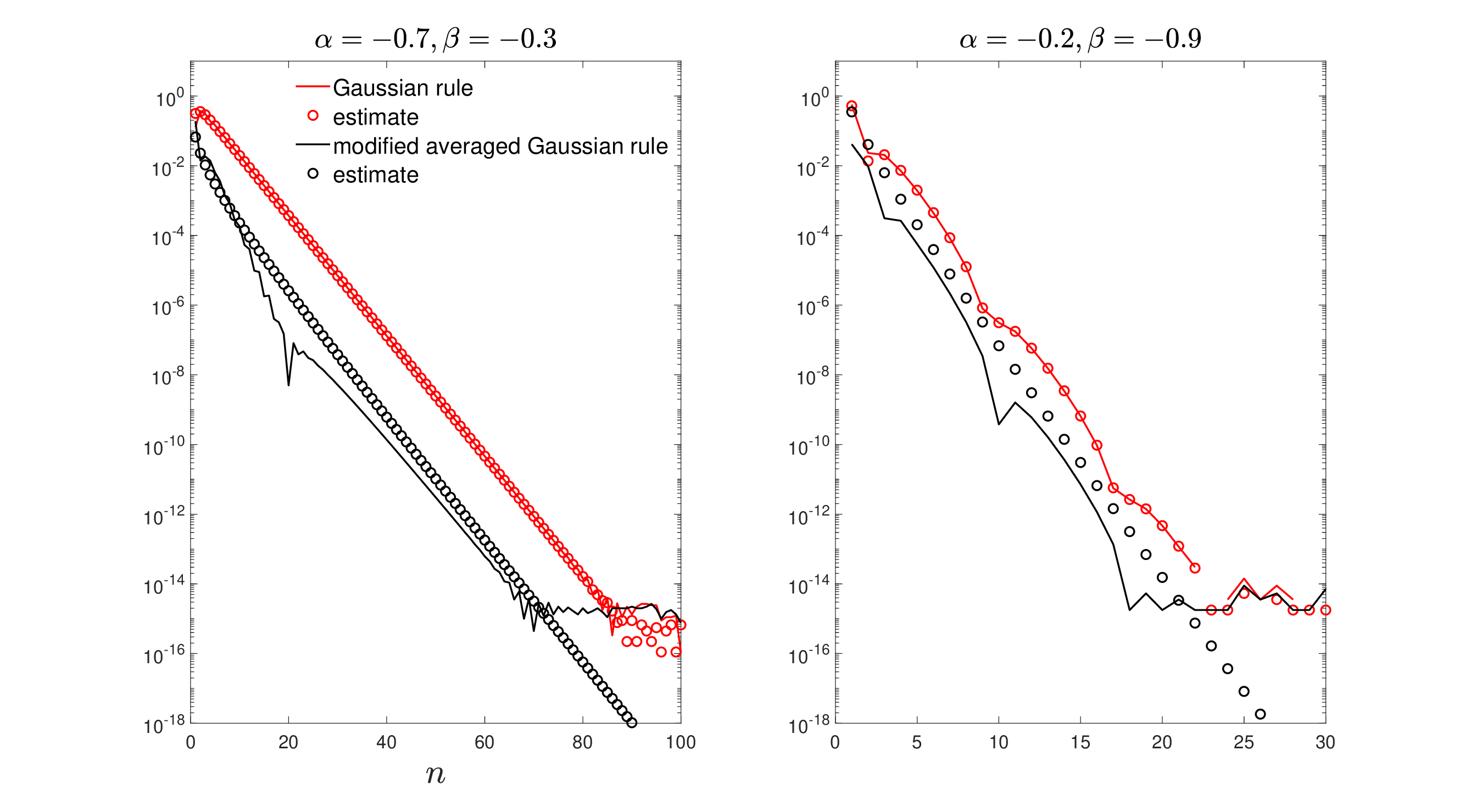} 
\caption{The error of the Gaussian rule (red line) with the a-posteriori estimate (\ref{stima a posteriori}) (red circles), and the error of the modified averaged Gaussian rule (black lines) with the a-priori estimate (\ref{bound errore}) (black circle), for the Jacobi weight. The functions considered are $f_1$ with $c=5$ (left) and $f_2$ with $d=3$ (right).}
\label{fig2_weighted}
\end{figure}

\section{Conclusions}

In this work we have presented a strategy to extend the internality of the anti-Gaussian Jacobi rule, for every couple $(\alpha,\beta)$, $\alpha,\beta >-1$.
The proper selection of the parameter $\theta=\theta(n)$ in the definition of the modified rule allows also to obtain a speed-up in the averaged formula. 
{\textcolor{black}{In the numerical section we have extended the analysis proposed in {\cite{DN26}} to the case of modified Gaussian rules with parameter $\theta(n) = \theta_{\mbox{\footnotesize{max}}}(n,\alpha,\beta)$, that is relevant when the classical averaged Gaussian rule is not of internal type, showing the advantages of this approach}}.

\vskip0.5cm 
\noindent
{\bf Acknowledgements.} 

Work partially supported by the DOR funds of the University of Padova and by the INdAM-GNCS 2026 projects. 
This research has been accomplished within the Community of Practice 
``Green Computing" of the Arqus European University Alliance, the RITA ``Research ITalian network on Approximation".


\end{document}